\documentclass[11pt,a4paper]{amsart}

\usepackage{amsmath,amssymb,amsthm,amscd,verbatim, enumerate}
\usepackage{graphicx}
\usepackage{epsfig}
\usepackage{hyperref}

\newtheorem{thm}{Theorem}
\newtheorem{lem}[thm]{Lemma}
\newtheorem{obs}[thm]{Observation}
\newtheorem{claim}[thm]{Claim}

\newtheorem{cor}[thm]{Corollary}
\newtheorem*{thm*}{Theorem}
\newtheorem{quest}[thm]{Question}

\theoremstyle{definition}

\theoremstyle{remark}

 \newcommand{\calT}{\mathcal{T}}

\newcommand{\calF}{\mathcal{F}}

\newcommand{\N}{\mathbb{N}}

\newcommand{\DEF}{\it}

\renewcommand{\le}{\leqslant}
\renewcommand{\leq}{\leqslant}
\renewcommand{\ge}{\geqslant}
\renewcommand{\geq}{\geqslant}

\begin{document}

\title[Relaxed $3$-Coloring of Planar Graphs]{Coloring
  Planar Graphs with three Colors and no Large Monochromatic
  Components}

\author{Louis Esperet} \address{Laboratoire G-SCOP (CNRS,
   Grenoble-INP) \\ Grenoble \\ France}
\email{louis.esperet@g-scop.fr}

\author{Gwena\"el Joret}
\address{Department of Mathematics and Statistics \\
  The University of Melbourne\\
  Melbourne\\
  Australia}

\email{gwenael.joret@unimelb.edu.au}

\thanks{Louis Esperet is partially supported by ANR Project
   Heredia under Contract \textsc{anr-10-jcjc-0204-01}. 
  Gwena\"el Joret is supported by a 
DECRA Fellowship from the Australian Research Council.}

\date{}
\sloppy

\begin{abstract}
We prove the existence of a function $f :\N \to \N$ 
such that the vertices of every planar graph with maximum
degree $\Delta$ can be $3$-colored in such a way that each
monochromatic component has at most $f(\Delta)$ vertices. This is best
possible (the number of colors cannot be reduced and the dependence on
the maximum degree cannot be avoided) and answers a question raised by
Kleinberg, Motwani, Raghavan, and Venkatasubramanian in 1997. 
Our result extends to graphs of bounded genus. 
\end{abstract}
\maketitle

\section{Introduction}

A proper vertex coloring of a graph $G$ is an assignment of colors to
the vertices of $G$ such that every color class is a stable set. In
other words, in each color class, connected components consist of
singletons. In this paper we investigate a relaxed version of this
classical version of graph coloring, where connected components in each
color class, called \emph{monochromatic components} in the rest of the
paper, have bounded size.

The famous HEX Lemma implies that in every 2-coloring of the triangular
$k\times k$-grid, there is a monochromatic path on $k$ vertices. This
shows that planar graphs with maximum degree $6$ cannot be 2-colored
in such a way that all monochromatic components have bounded size. On
the other hand, Haxell, Szab\'o and Tardos~\cite{HST03} proved that
every (not necessarily planar) graph with maximum degree at most $5$ can be
2-colored in such a way that all monochromatic components have size at
most 20000. This bound was later reduced to 1908 by Berke~\cite{B08}.

As for three colors, Kleinberg, Motwani, Raghavan, and
Venkatasubramanian~\cite[Theorem 4.2]{KMRV97} constructed planar
graphs that cannot be $3$-colored in such a way that each
monochromatic component has bounded size.  However, their examples
have large maximum degree, which prompted them to ask the following
question.

\begin{quest}{\cite[Question 4.3]{KMRV97}}\label{quest}
Is there a function $f:\N \to \N$  
such that every planar graph with maximum degree
at most $\Delta$ has a 3-coloring in which each monochromatic
component has size at most $f(\Delta)$?
\end{quest}

A similar construction was given by Alon, Ding, Oporowski and
Vertigan~\cite[Theorem 6.6]{ADOV03}, who also pointed that they do not
know whether examples with bounded maximum degree can be
constructed. Question~\ref{quest} was also raised more recently by
Linial, Matou\v sek, Sheffet and Tardos~\cite{LMST08}.  Our main
result is a positive answer to this question.

\begin{thm}
\label{thm:main_intro} 
There exists a function $f:\N \to \N$ such that every
planar graph with maximum degree $\Delta$ has a 3-coloring
in which each monochromatic component has size at most $f(\Delta)$.
\end{thm}

This theorem will be proved in Section~\ref{sec:main}. Let us remark
that we prove Theorem~\ref{thm:main_intro} with a rather large
function $f$, namely $f(\Delta)=(15\Delta)^{32\Delta+8}$, which is
almost surely far from optimal. We have strived to make our proofs as
simple as possible, and as a result we made no effort to optimize the
various bounds appearing in the paper.

In Section~\ref{sec:surfaces}, we extend Theorem~\ref{thm:main_intro}
to graphs embeddable in a fixed surface. This improves a special case
of a result of Alon, Ding, Oporowski, and Vertigan~\cite{ADOV03}, who
proved that for every proper minor-closed class of graphs $\mathcal G$, 
there is 
a function $f_{\mathcal G}:\N \to \N$ 
such that every graph in $\mathcal G$ with
maximum degree $\Delta$ can be 4-colored in such way that each
monochromatic component has size at most $f_{\mathcal G}(\Delta)$.

Finally, in Section~\ref{sec:Conclusion}  
we conclude with some remarks and open problems.

\section{Preliminaries}
\label{sec:preliminaries}

All graphs in this paper are finite, undirected, and simple.  We
denote by $V(G)$ and $E(G)$ the vertex and edge sets, respectively, of
a graph $G$.  We use the shorthand $|G|$ for the number of vertices of
a graph $G$ and denote the maximum degree of $G$ by $\Delta(G)$. 
We let $\deg_{G}(v)$ denote the degree of vertex $v$ in $G$. 

The term `coloring' will always refer to a vertex coloring
of the graph under consideration. For simplicity, we identify colors 
with positive integers, and we let a $k$-coloring be a coloring using  
colors in $\{1,2, \dots, k\}$. 
Note that we do not require a
coloring to be proper, that is, adjacent vertices may receive the same
color. Given a coloring $\phi$ of $G$, a {\DEF monochromatic
  component} is a connected component of the subgraph of $G$ induced
by some color class. A monochromatic component of color $i$ is also
called an {\DEF $i$-component}. The {\DEF size} of a component is its
number of vertices.

\section{Proof of the main Theorem}\label{sec:main}

We start with a brief sketch of the proof of
Theorem~\ref{thm:main_intro}. We consider a decomposition of the
vertex set of our planar graph $G$ drawn in the plane into sets
$O_1,O_2,\ldots,O_k$, each inducing an outerplanar graph. The set
$O_1$ is the vertex set of the outerface of $G$, and for $i=2, \dots,
k$, the set $O_i$ is the vertex set of the outerface of the subgraph
of $G$ induced by $V(G)\setminus (\bigcup_{1\le j\leq i-1} O_j)$.

We color the graph $G$ with colors $1,2,3$ in such way that for each
$i \in \{1,\ldots,k\}$, no vertex of $O_i$ has color $1+ (i \bmod
3)$. This implies that each monochromatic component is contained in
the union of two consecutive sets $O_i$ and $O_{i+1}$. Starting with
$O_{k}$, we color the sets $O_i$ one after the other in decreasing
order of their index $i$.  Given a coloring of $O_{i+1}$, we extend
this coloring to a coloring of $O_{i+1} \cup O_{i}$.  This extension
is done so as to maintain the property that in one of the two color
classes of $O_i$, monochromatic components are particularly small;
thus the two colors do not play symmetric roles, one is `small' and
the other `large'.  The small color of $O_{i+1}$ then becomes the
large color of $O_i$, while the large color of $O_{i+1}$ does not
appear at all in $O_i$.

While the above approach is natural, we found that making it
work required to carefully handle a number of situations. In
particular, we were led to introduce a technical lemma,
Lemma~\ref{lem:blackbox_with_conditions} below, whose proof might
appear a bit uninviting to the otherwise interested
reader. We hope the reader will bear with us till the main part of the
argument, which is provided by Theorem~\ref{thm:main}.

\begin{lem}\label{lem:blackbox_with_conditions}
Let $G$ be a connected plane graph whose vertex set is partitioned
into an induced path $P$ on at least $3$ vertices, 
and a stable set $S$ with a distinguished vertex $r$. 
Let $d$ be the maximum degree of a vertex
in $P$, and let $\Delta := \Delta(G)$.  Assume further that

\begin{itemize}
\item $r$ is adjacent to the two endpoints of $P$ and no other vertex of $P$;
\item the outerface of $G$ is bounded by the chordless cycle $G[V(P)
  \cup \{r\}]$;
\item every vertex in $S$ has degree at least $2$;
\item if $u\in S$ has degree exactly $2$, then the two neighbors of
  $u$ on $P$ are not adjacent; and
\item every two consecutive vertices of $P$ have at least one
  common neighbor in $S$.
\end{itemize}
Then there exists a $2$-coloring of $G$ in which the two
endpoints of $P$ and all the vertices in $S$ have color $2$, each
$1$-component has size at most $2d+1$, and each $2$-component has
size at most $(3\Delta)^{3d-4}$.
\end{lem}
\begin{proof}
First we need to introduce a number of definitions and notations.  We
think of the path $P$ as being drawn horizontally in the plane with the 
vertices of $S$ above $P$; thus 
the vertices of $P$ are ordered from left to right. This ordering induces in
a natural way a linear ordering of every subset $X\subseteq V(P)$.
Two vertices of such a subset $X$ are said to be {\DEF consecutive in
  $X$} if they are consecutive in this ordering.  Let $x$ and $y$
denote the left and right endpoint, respectively, of the path $P$.

For simplicity, the color {\DEF opposite to} 1 is defined to be 2, and
vice versa.  Consider a subset $X \subseteq V(P)$ with $|X| \geq 2$
and call $a$ and $b$ the leftmost and rightmost vertices of $X$,
respectively. If $a$ and $b$ are colored, each either 1 or 2,
but no vertex in $X \setminus \{a, b\}$ is colored, then an {\DEF
  $\{a,b\}$-alternate coloring} of $X$ consists in keeping the colors
on $a, b$, and coloring the vertices of $X \setminus \{a, b\}$ (if
any) as follows. We enumerate the vertices of $X$ from left to right
as $a, x_{1}, \dots, x_{k}, b$.  If $k=1$, then $x_{1}$ is colored
with color 2 if both $a$ and $b$ have color 1; otherwise, $x_{1}$ is
colored with color 1.  If $k\geq 2$, then $x_{1}$ and $x_{k}$ are
colored with the color opposite to that of $a$ and $b$, respectively,
and for each $i\in \{2, \dots, k-1\}$, the vertex $x_{i}$ is colored
with the color opposite to that of $x_{i-1}$.  Let us point out some
simple but useful properties of this coloring:
\begin{itemize}
\item no three consecutive vertices in $a, x_{1}, \dots, x_{k}, b$
  have the same color;
\item if $k\geq 1$ and $a$ has color 2, then $x_{1}$ has color
  1, and
\item if $k\geq 1$ and $b$ has color 2, then $x_{k}$ has color 1. 
\end{itemize}
These properties will be used repeatedly, and sometimes implicitly, in
what follows.

Let $\calF$ be the set of bounded faces of $G$.  For $f \in \calF$,
let $\partial f$ denote the subgraph of $G$ which is the boundary of
$f$.  We note that, because of our assumptions on $S$, every edge of
$P$ is included in the boundary of a triangular face of $\calF$.

Let $\rho$ denote the unique bounded face of $G$ which includes the
vertex $r$ in its boundary.  We define a rooted tree $T$ with vertex
set $\calF$ and root $\rho$ inductively as follows. First, let
$s(\rho) = r$ and let $$S(\rho)=(S\cap V(\partial \rho)) \setminus
\{r\}.$$ Let the children of $\rho$ in $T$ be the faces distinct from
$\rho$ that are incident to some vertex in $S(\rho)$.  Now, consider a
face $f \in \calF \setminus \{\rho\}$ with parent $f^*$ in $T$.  Let
$s(f)$ be the unique vertex of $S$ included in $V(\partial f) \cap
V(\partial f^*)$ (the existence and uniqueness of $s(f)$ will be
proved below). Let $$S(f)=(S\cap V(\partial f))
\setminus \{s(f)\}.$$ The children of $f$ are then all the faces $f'\ne f$
incident to a vertex of $S(f)$. 

In order to show that $T$ is well defined, we only
need to prove that the vertex $s(f)$ defined above exists and is
unique. The existence follows from the definition of $T$, since $f$
and $f^*$ share a vertex of $(S\cap V(\partial f^*))\setminus
\{s(f^*)\}$. Note also that for any vertex $v \in (S\cap V(\partial
f^*))\setminus \{s(f^*)\}$, the vertices $v^-$ and $v^+$ just preceding
and following $v$ in a boundary walk of $f^*$ lie both on $P$ and any
face incident to $v$ distinct from $f^*$ is inside the region bounded
by $v,v^-,v^+$ and the subpath of $P$ between $v^-$ and $v^+$. It
follows that $s(f)$ is unique.

The {\DEF depth} $dp(f)$ of a face $f\in \calF$ 
is its depth in $T$, the root $\rho$ having depth $0$.
Observe that, because of our assumptions on $S$, the leaves of $T$ are
precisely the triangular faces sharing an edge with the
outerface. Observe also that $T$ can be equivalently defined as the
(unique) breadth-first search tree rooted at $\rho$ of the graph with
vertex set $\calF$ in which two vertices $f,f'\in \calF$ are adjacent
if the corresponding faces in $G$ share a vertex of $S$.

A face $f \in \calF$ is uniquely determined by, and uniquely
determines, the triplet $[a,s,b]$ where $s=s(f)$ and $a,b$ are the
leftmost and rightmost neighbors of $s$ on $P$ included in $\partial
f$, respectively.  With a slight abuse of notation, we write
$f=[a,s,b]$ to denote the face $f$ with triplet $[a,s,b]$.

We define the following sets of vertices associated to a face
$f=[a,s,b] \in \calF$ (see Figure~\ref{fig:bb} for an illustration). 
The set $$\Sigma(f)=(V(P)\cap V(\partial f)) \setminus \{a,b\}$$ 
is the set of {\DEF corners} of $f$, and 
$$\Pi(f)=N(S(f)) \setminus (\Sigma(f) \cup \{a, b\})$$ is the set of
{\DEF pivots} of $f$. Here, $N(X)$ denotes the set of vertices of
$V(G) \setminus X$ having a neighbor in $X$.  

Observe that the sets $\Sigma(f), \Sigma(f'), \Pi(f), \Pi(f')$ are
pairwise disjoint for every two distinct faces $f, f' \in \calF$.
Moreover, $\bigcup_{f \in \calF} (\Sigma(f) \cup \Pi(f)) = V(P)
\setminus \{x, y\}$.  Thus every {\em internal} vertex of the path $P$
is either a corner or a pivot of some uniquely determined face $f$,
which we denote by $f(v)$.  When $v$ is a pivot, the unique neighbor
of $v$ in $S$ that is incident to $f(v)$ is denoted by $\psi(v)$.  For
each vertex $v\in S \setminus \{r\}$, let similarly $f(v)$ denote the
unique face $f \in \calF$ such that $v\in S(f)$.  

\begin{figure}[htbp]
\centering
\includegraphics[scale=1]{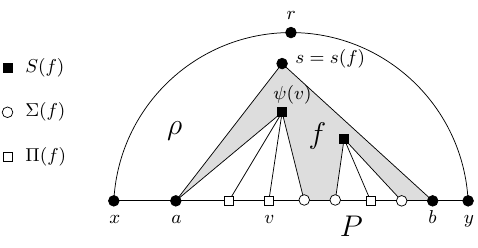}
\caption{A face $f=[a,s,b]$ and the corresponding sets $S(f)$,
  $\Sigma(f)$ and $\Pi(f)$.} \label{fig:bb}
\end{figure}

\smallskip

Consider two faces $f=[a,s,b]$ and $f'$ such that $f'$ is inside 
the cycle formed  
by the edges $as, bs$ and the path from $a$ to $b$ on $P$. 
Note that $f$ is an ancestor of $f'$ in $T$. The following
observation describes precisely the subgraph of $T$ induced by all the 
faces incident to a given internal vertex of $P$ (see Figure~\ref{fig:ne} for
an illustration).

\begin{obs}
\label{obs:crucial}
Let $w$ be an internal vertex of $P$. Let $u_{1}, \dots, u_{k}$ be the
neighbors of $w$ in clockwise order around $w$, with $u_{1}$ and
$u_{k}$ the left and right neighbors, respectively, of $w$ on $P$.
For each $i\in \{1, \dots, k-1\}$, let $f_{i}$ be the unique face in
$\calF$ with $wu_{i}, wu_{i+1} \in E(\partial f_{i})$.
\smallskip

(a) If $w$ is a corner of $f_{j}$ for some $j \in \{1, \dots, k-1\}$,
then $f_1,f_2,\ldots, f_{k-1}$ is a path in $T$, with $f_j=f(w)$ being the
face of smallest depth.  In particular, $dp(f_i) - dp(f(w))\le d-2$
for each $i \in \{1, \dots, k-1\}$.

\smallskip

(b) If $w$ is a pivot with $\psi(w)=u_{j}$, then $j\in \{2, \dots, k-1\}$ and 
$f_1,\ldots, f_{j-1},
f(w), f_{j},\ldots, f_{k-1}$ is a path in $T$, with $f(w)$ being the
face of smallest depth.  In particular, $dp(f_i) - dp(f(w))\le d-2$
for each $i \in \{1, \dots, k-1\}$.
\end{obs}

\begin{figure}[htbp]
\centering
\includegraphics[scale=1]{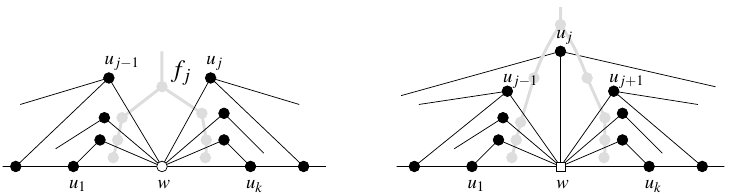}
\caption{The two configurations in Observation~\ref{obs:crucial}. The
  tree $T$ is depicted in gray.} \label{fig:ne}
\end{figure}

An internal vertex $v$ of $P$ is said to be an {\DEF isolated pivot} if 
$v$ is a pivot, $\deg_{G}(v) = 3$, 
and the two faces in $\calF$ incident to $v$ are triangular. 

With these definitions in hand, we may now describe our coloring of
the graph $G$. First, recall that the vertices in $S$ must be colored 
with color 2. So it remains to color the vertices of $P$. 
These vertices are colored as follows. We perform a
depth-first walk in $T$ starting from its root $\rho$, 
and for each face $f \in \calF$ encountered
we color the vertices in $\Pi(f)$ and $\Sigma(f)$. This ensures that,
when considering a face $f=[a,s,b]$ distinct from the root $\rho$, the
two vertices $a$ and $b$ are already colored.  Given $f=[a,s,b]$,
\begin{itemize}
\item if $f = \rho$, we color both $x$ and $y$ with color 2;
\item if $f \neq \rho$, 
we perform an $\{a, b\}$-alternate coloring of $\Sigma(f)\cup \{a,b\}$;
\item we color each isolated pivot in $\Pi(f)$ with color 2, and
\item we color each non-isolated pivot in $\Pi(f)$ with color 1 if $dp(f) \bmod
  2d \in \{0, \ldots, d-1\}$, and with color 2 otherwise.
\end{itemize}

Let us consider the maximum size of monochromatic components in this
coloring of $G$, starting with color 1. Since all vertices in $S$ and
the two endpoints $x$ and $y$ of $P$ have color 2, each 1-component of
$G$ is a subpath of $P \setminus \{x,y\}$.  We define a {\DEF $1$-path}
as a (not necessarily maximal) subpath of $P \setminus \{x,y\}$, every
vertex of which has color 1.

\begin{claim}
\label{cl:nonconsecutive}
If $Q$ is a $1$-path, then
each vertex in $S$ has at most two neighbors on $Q$, in which case
they are consecutive vertices of $Q$.
\end{claim}
\begin{proof}
Let $w_{1}, \dots, w_{k}$ be the vertices of $Q$ enumerated 
from left to right. Arguing by contradiction, suppose there exists
$u\in S$ adjacent to $w_{i}$ and $w_{j}$ with $i + 1 < j$, and choose
such a triple $(u, w_{i}, w_{j})$ with $j - i$ minimum, and with
respect to this, $dp(f(u))$ maximum. The vertices $w_{i}$ and
$w_{i+1}$ have a common neighbor $u' \in S$. If $u=u'$, then the
triple $(u, w_{i+1}, w_{j})$ is a better choice than $(u, w_{i},
w_{j})$, unless $j=i+2$, in which case $w_{i+1}$ is an isolated pivot
and has color 2 (here we use the fact that there cannot be any vertex 
of $S$ inside the cycles $uw_{i}w_{i+1}$ and $uw_{i+1}w_{i+2}$ since 
such a vertex would have degree exactly $2$ and would be adjacent to two 
consecutive vertices of $P$). Thus we obtain a contradiction in both cases, 
and hence, $u\ne u'$. 
It follows that $dp(f(u'))>dp(f(u))$ since $u'$ is inside the cycle consisting of the edges  
$uw_i,uw_j$ and the subpath of $Q$ between $w_i$ and $w_j$. 
Now, the vertex $u'$ cannot have degree exactly $2$, and thus 
$u'w_{\ell} \in E(G)$ for some $\ell \in \{i+2, \dots, j\}$. However, 
the triple  $(u', w_i, w_\ell)$ is then a better choice than 
 $(u, w_{i}, w_{j})$, a contradiction (indeed, either $\ell < j$, or 
 $\ell=j$ but $dp(f(u'))>dp(f(u))$). 
\end{proof}

We deduce that 1-components have bounded size.

\begin{claim}
\label{cl:not_inside}
Every $1$-path has at most $2d+1$ vertices. 
\end{claim}
\begin{proof}
Arguing by contradiction, suppose that $Q$ is a $1$-path with $2d+2$ vertices, and let $w_{1}, \dots, w_{2d + 2}$ be its vertices
enumerated from left to right. Let $u_{1} \in S$ be a common neighbor
of $w_{d + 1}$ and $w_{d + 2}$. By Claim~\ref{cl:nonconsecutive},
$w_{d+1}$ and $w_{d+2}$ are the only neighbors of $u_{1}$ on
$Q$, therefore $u_1$ has a neighbor in $V(P)\setminus V(Q)$ by our
assumption on $S$. Let $v_{1}$ be such neighbor at minimum distance
from $w_{d+ 2}$ on $P$. Either $v_{1}$ is on the right of $Q$ or
on the left of $Q$; since $w_{d+1}, w_{d+2}$ are the two middle
vertices of $Q$, these two cases are symmetric, and thus we may assume
without loss of generality that $v_{1}$ is on the right of $Q$.  Then
$[w_{d+2},u_1,v_1]$ is a face distinct from the root face
$\rho$. Let $z$ be the right neighbor of $w_{2d+2}$ on $P$ (thus
$z\notin V(Q)$), and let $A$ denote the $z$--$v_{1}$ subpath of $P$.

\begin{figure}[htbp]
\centering
\includegraphics[scale=1]{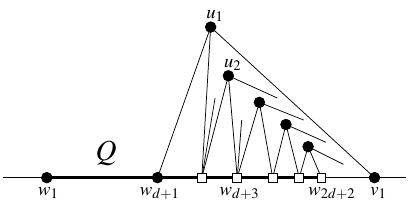}
\caption{Illustration of the proof of
  Claim~\ref{cl:not_inside}.} \label{fig:gp}
\end{figure}

For $i =2, \dots, d+1$, let $u_{i} \in S$ be a common neighbor of
$w_{d + i}$ and $w_{d + i + 1}$ (see Figure~\ref{fig:gp}). By
Claim~\ref{cl:nonconsecutive}, the vertices 
$u_{1}, \dots, u_{k}$ are all distinct, and thus
each such vertex has a neighbor in $V(P) \setminus V(Q)$, which must
then be on $A$ because of the face $[w_{d+2},u_1,v_1]$. Moreover,
for each $i \in \{1, \dots, d+1\}$, we have that $w_{d+i+1}$ is a
pivot, and $u_i=\psi(w_{d+i+1})$. For each such index $i$, let
$f_i=f(w_{d+i+1})=f(u_i)$. It follows from Claim~\ref{obs:crucial}
that $1\le dp(f_{i+1})-dp(f_{i})\le d-2$. This in turn implies 
that there exists an index $i\in \{1, \dots, d+1\}$ such that 
$dp(f_{i}) \bmod 2d \in \{d,\ldots,2d-1\}$. But then the pivot vertex 
$w_{d+i+1}$ was colored $2$ in our coloring of $G$, a contradiction. 
\end{proof}

We now bound the size of monochromatic components of color 2. Let thus
$K$ be a 2-component of $G$.  We start by gathering a few observations
about $K$.

Observe that, if $f \in \calF$ with $f=[a,s,b]$, then $\{a,b\}$
separates all vertices $v$ such that $v \in S(f')\cup \Sigma(f') \cup
\Pi(f')$ for some face $f'$ that is a descendant of $f$ in $T$ from
the remaining vertices of $G$. (Note that $f$ is considered to be a
descendant of itself). It follows:

\begin{obs}
\label{obs:blue_connected}
Let $f \in \calF$ with $f = [a,s,b]$,  
and let $K_f$ be the set of vertices $v \in V(K)$ such
that $v \in S(f')\cup \Sigma(f') \cup \Pi(f')$ for some face $f'$ that
is a descendant of $f$ in $T$. If there are two vertices $u,v \in V(K)$
with $u \in V(K_f)$ and $v\not\in V(K_f)$, then at least one of $a,b$ is in
$K$.
\end{obs}

Let $\calF_K$ be the set of faces $f\in \calF$ such that $S(f)\cup
\Sigma(f) \cup \Pi(f)$ contains a vertex of $K$, and let $T_K$ denote
the subgraph of $T$ induced by $\calF_K$. Suppose that $T_K$ is not
connected. Then $\calF_K$ contains two faces $f=[a,s,b]$ and $f'$ such
that the parent $f^*$ of $f$ is not in $\calF_K$ and $f'$ is not a
descendant of $f$.  By Observation \ref{obs:blue_connected}, this
implies that at least one of $a,b$ is in $K$, and consequently $s\in
V(K)$. Since $s\in S(f^*)$, we deduce that $f^* \in \calF_K$, a
contradiction. It follows:

\begin{obs}
\label{obs:blue_tree}
$T_K$ is a subtree of $T$.
\end{obs}

Let $\widetilde f$ be the face in $\calF_{K}$ having smallest depth in
$T$. We see $T_{K}$ as being rooted at $\widetilde f$. Our aim now
is to bound the height of $T_K$.

\begin{claim}
\label{claim:depth}
$T_K$ has height at most $3d-5$.
\end{claim}
\begin{proof}
Let $f_1$ be a leaf of $T_K$ farthest from $\widetilde f$. We may
assume that $f_1\ne \widetilde f$, since otherwise $T_K$ has height 0
and the claim trivially holds.  Let $A_K$ be set of ancestors of
$f_1$ in $T_K$, $f_{1}$ included.  Thus $A_K$ induces a path in $T_K$
with endpoints $f_1$ and $\widetilde f$.  Starting with $f_{1}$, we
define inductively a sequence $f_{1}, f_{2}, \dots, f_{t}$ of faces,
with $f_i=[a_i,s_i,b_i]$ and $f_{i}\in A_K$ for each $i \in \{1,
\dots, t\}$, as follows. For $i \geq 2$, if $f_{i-1}$ is distinct from
$\widetilde f$, then by Observation~\ref{obs:blue_connected}, at least
one of $a_{i-1},b_{i-1}$ is in $K$. Let $h_{i-1}$ denote such a
vertex, and let $f_i=f(h_{i-1})$.  If $f_{i-1} = \widetilde f$ then
$f_{i}$ is not defined, and $f_{i-1}=f_{t}$ becomes the last face in
the sequence.

Let $i\in\{2, \dots, t\}$. By definition of $T_K$, the face $f_i$ is
in $T_K$. By Observation~\ref{obs:crucial}, $f_i$ is an ancestor of
$f_{i-1}$, which implies inductively that $f_i \in A_K$ (since $f_1
\in A_K$).  Moreover, $dp(f_i) < dp(f_{i-1})$ since $f_i\ne
f_{i-1}$. Since $f_i=f(h_{i-1})$ and $f_{i-1}$ is incident to
$h_{i-1}$, Observation~\ref{obs:crucial} also implies that
$dp(f_{i-1})\le dp(f_{i})+d-2$.

Let $i \in \{2, \dots, t-1\}$. The vertex $h_{i}$ has to be connected
to $h_{i-1}$ by a path in $K$.  It follows from the
definition of an $\{a_i, b_i\}$-alternate coloring that $h_{i-1}$
cannot be a corner of $f_i$. (In fact, this is the key property of an
$\{a_i, b_i\}$-alternate coloring.)  Therefore, $h_{i-1}$ is a pivot of
$f_i$. Since  $S(f_{i-1})\cup \Sigma(f_{i-1}) \cup \Pi(f_{i-1}) \neq \varnothing$, the face $f_{i-1}$ is not triangular, and hence  
$h_{i-1}$ is not an isolated pivot. It follows
that $dp(f_i)\bmod 2d \in \{d, \ldots, 2d-1\}$.

Write $dp(f_2)=2kd+\ell_2$, with $d\le \ell_2 \le 2d-1$, and for each
$i \in \{3, \dots, t-1\}$, let $\ell_i=dp(f_i)-2kd$. Since $dp(f_{i})
< dp(f_{i-1})\le dp(f_{i})+d-2$ for each $i \in \{2, \dots, t-1\}$,
and $dp(f_i)\bmod 2d \in \{d, \ldots, 2d-1\}$, we have $d\leq
\ell_{t-1} < \ell_{t-2} < \cdots < \ell_{2} \leq 2d-1$. In particular,
$\ell_{2}-\ell_{t-1}\le d-1$. Now, the height of $T_K$ is precisely
$dp(f_1)-dp(f_t)=\sum_{i=2}^t
(dp(f_{i-1})-dp(f_{i}))=dp(f_1)-dp(f_{2})+\ell_{2}
-\ell_{t-1}+dp(f_{t-1})-dp(f_{t})$. By Observation~\ref{obs:crucial},
$dp(f_{t-1})-dp(f_{t})\le d-2$ and $dp(f_1)-dp(f_{2})\le d-2$. Using
that $\ell_{2}-\ell_{t-1}\le d-1$, we obtain that the height of $T_k$
is at most $3d-5$.
\end{proof}

Consider a face $f=[a,s,b]$ of $T_K$. By the definition of an $\{a,
b\}$-alternate coloring, there are at most two consecutive vertices
with color 2 in $\Sigma(f)$. Therefore, using
Observation~\ref{obs:blue_connected}, $\Sigma(f)$ contains at most 2
vertices of $K$ and $S(f)$ contains at most 3 vertices of $K$. It
follows that $|K\cap \Pi(f)|\leq 3(\Delta-2)$, which implies that
$S(f)\cup \Sigma(f) \cup \Pi(f)$ contains at most $3\Delta-1$ vertices
of $K$. Also, we deduce that $f$ has at most $3\Delta$ children in $T_K$.  
Using Claim~\ref{claim:depth}, we then obtain
$$|T_K|\leq \sum_{i=0}^{3d-5} (3\Delta)^{i} = \frac{(3\Delta)^{3d-4}}{3\Delta-1}$$ and hence $|K|\leq(3\Delta-1)|T_K| \leq (3\Delta)^{3d-4}$, as desired. 
This concludes the proof of Lemma~\ref{lem:blackbox_with_conditions}.
\end{proof}

At the expense of a slightly larger bound on the size of monochromatic
components, we may relax the requirements in
Lemma~\ref{lem:blackbox_with_conditions} as follows.

\begin{lem}\label{lem:blackbox}
Let $G$ be a connected plane graph whose vertex set is partitioned
into a chordless cycle $C$ and a stable set $S$ such that the cycle
$C$ bounds a face of $G$.  Let $d$ be the maximum degree of a vertex
in $C$, and let $\Delta := \Delta(G)$.  Then there exists a
$2$-coloring of $G$ in which each vertex in $S$ has color $2$, each
$1$-component has size at most $2d+5$, and each $2$-component has
size at most $d(6\Delta)^{3d+2}$.
\end{lem}
\begin{proof}
We may assume without loss of generality that $C$ bounds the outerface
of $G$.  Let $S^*$ be the set of vertices $v\in S$ such that either
$\deg_G(v)\le 1$, or $\deg_G(v)=2$ and the two neighbors of $v$ are
adjacent. Let $G^*=G\setminus S^*$, and remove from $S$ the vertices
in $S^{*}$.  (We will treat the vertices in $S^{*}$ at the very end.)
We construct a new graph $G'$ from $G^*$ in two steps as follows.

\smallskip
\noindent {\it Step 1.} Take a maximal stable set $Z$ of the vertices
$\{v\in V(C),\, \deg_{G^*}(v)=2\}$ ($Z$ might be empty), and for each vertex 
$v\in Z$ add a vertex $s_v$ in $S$ adjacent to $v$ and its two
neighbors in $C$.

Note that this can be done so that the embedding stays planar and $C$
still bounds the outerface of the graph. By our choice of $Z$, after
Step 1 every vertex of $C$ has degree at least 3 (and thus, has at
least one neighbor in $S$).

\smallskip

\noindent {\it Step 2.} For each pair of consecutive vertices $u,v$ in
$C$ in anti-clockwise order having no common
neighbor in $S$, do the following. Let $f$ be the inner face incident
to $uv$, and let $s$ be the unique neighbor of $u$ in $S$ that is
incident to $f$. Add an edge between $s$ and $v$.

\smallskip

Again, this can be done so that the embedding stays planar and $C$
still bounds the outerface of the graph. 
(How the degrees of vertices increased will be considered later.) 
Let $G'$ be the graph obtained after Step 2. Note that $G'$ is a supergraph of $G^*$. 

Let $x,y$ be two arbitrarily chosen consecutive vertices of $C$, 
and let $P$ denote
the $x$--$y$ path in $C$ that avoids the edge $xy$. Subdivide the edge
$xy$ by adding a vertex $r$ between $x$ and $y$, and add $r$ to $S$.
Observe that the graph $G''$ obtained after this operation together
with the set $S$ satisfy the assumptions of
Lemma~\ref{lem:blackbox_with_conditions}. Indeed, $P$ is an induced path in
$G''$ with endpoints $x$ and $y$, and $r\in S$ is only adjacent to
$x,y$, while all other vertices in $S$ are inside 
the cycle induced by $V(P)\cup \{r\}$. Moreover, 
every vertex in $S$ has degree at least $2$; 
if $u\in S$ has degree exactly $2$, then the two neighbors of 
 $u$ on $P$ are not adjacent, and every two consecutive vertices of $P$ have
 at least one common neighbor in $S$.

The degree of each vertex of
$P$ increased by at most two during Steps 1 and 2, while the degree of
each vertex in $S$ can at worst be doubled at Step 2 (we might add an
edge $sv$ for every neighbor $u$ of $s$). It follows that $G''$ has
maximum degree at most $2\Delta$ and vertices of $P$ have degree at
most $d+2$. By Lemma~\ref{lem:blackbox_with_conditions}, $G''$ has a
$2$-coloring such that $1$-components have size at most
$2d+5$, $2$-components have size at most $(6\Delta)^{3d+2}$, and
$x,y,r$ have color 2 (in particular, $x$ and $y$ are in the same
2-component).

We now add back the edge $xy$ and the vertices of $S^*$, which we
connect to their original neighbors in $G$, and color them with
color 2.  By the definition of $S^*$ and the remark above, this does
not connect different 2-components of $G''$. Since each vertex of $C$
had at most $d$ neighbors in $S^*$, in the resulting graph $G'''$ the
size of each 2-component is at most $d(6\Delta)^{3d+2}$, while
1-components still have size at most $2d+5$ since they 
remain unchanged. Since the graph $G$ is a subgraph of $G'''$, 
these two bounds obviously hold for this coloring 
restricted to $G$. 
\end{proof}

Let $g_1: \N \to \N$ and $g_2: \N\times \N \to \N$ denote the bounds
on the sizes of 1- and 2-components, respectively, appearing in
Lemma~\ref{lem:blackbox}; namely $g_1(d) := 2d + 5$ and $g_2(d,\Delta)
:= d(6\Delta)^{3d+2}$.

For a plane graph $G$, we denote by $O(G)$ the set of vertices lying 
on the boundary of the outerface of $G$, and by $O_2(G)$ the set of vertices not in $O(G)$
that are adjacent to a vertex in $O(G)$. A plane graph is {\DEF
  near-triangulated} if all its faces are triangular, except possibly for
the outerface.  Note that if $G$ is near-triangulated, then $O_2(G)$
is precisely the set of vertices on the outerface of $G \setminus O(G)$.

We will use the following simple observation. 
\begin{obs} 
\label{obs:recolor}
Let $\ell\ge 1$ be an integer. Suppose we have a coloring of a
graph with maximum degree at most $\Delta\geq 1$ 
in which every $i$-component has
size at most $k$, for some color class $i$. Then, if we recolor at
most $\ell$ vertices of the graph, in the new coloring every
$i$-component has size at most $\ell \Delta k+\ell\le 2 \ell \Delta
k$.
\end{obs}

We now use Lemma~\ref{lem:blackbox} to prove the following result by
induction. 

\begin{thm}\label{thm:main} 
Every connected near-triangulated plane graph $G$ with maximum degree
at most $\Delta\geq 1$ has a $3$-coloring such that
\begin{enumerate}[(i)]
\item no vertex of $O(G)$ is colored with color $3$; 
\item no vertex of $O_2(G)$ is colored with color $1$; 
\item each $1$-component intersecting $O(G)$ has size at most
  $f_1(\Delta)=16\Delta^2 g_1(\Delta)$;
\item each $2$-component intersecting $O(G)\cup O_2(G)$ has size at
  most $f_2(\Delta)=16\Delta^2 f_1(\Delta) \,g_2(\Delta,\Delta
  \,f_1(\Delta))$, and
\item each monochromatic component has size at most $6 \Delta^2
  f_2(\Delta)$.
\end{enumerate}
\end{thm}
\begin{proof}
We prove the theorem by induction on $|G|$.  The proof is split into
five cases, depending on the structure of the outerplanar graph $J$
induced by $O(G)$. In fact, to make the induction work, we will need
to prove additional properties in some of the cases.  Instead of
stating here the exact statement that we prove by induction (which
would be lengthy), we describe at the beginning of each case below
what are the extra properties we wish guarantee in that case (if any).

\medskip

\paragraph{\bf Case 0: $|G| =1$} This is the base case of the induction, 
which trivially holds. Let us now consider the
inductive case $|G| > 1$. 

\medskip

\paragraph{\bf Case 1: $G$ has a vertex of degree one} Let $v$ be such
a vertex. Since $G$ is near-triangulated, $v$ and its neighbor $u$
both lie on the boundary of the outerface. We can color $G\setminus v$
by induction and assign to $v$ a color (1 or 2) different from that of
$u$. This does not affect the sizes of existing monochromatic components, and the newly created monochromatic component has size $1$. Thus the 
resulting coloring of $G$ satisfies conditions (i)--(v).  
In the rest of the proof we assume
that $G$ has minimum degree at least two.

\medskip

\paragraph{\bf Case 2: The outerplanar graph $J$ is a chordless cycle}  
In this case we show a strengthened version of (iii) and (iv) where a
multiplicative factor of $16\Delta^2$ is saved in the bounds, as well
as a better bound for 3-components intersecting $O_2(G)$:
\begin{enumerate}[(a)]
\item each 1-component intersecting $O(G)$ has size at most
  $g_1(\Delta)$;
\item each 2-component intersecting $O(G)\cup O_2(G)$ has size at
  most $f_1(\Delta) \,g_2(\Delta,\Delta  \,f_1(\Delta))$, and
\item each 3-component intersecting $O_2(G)$ has size at
  most $f_2(\Delta)$. 
\end{enumerate}

Since $G$ is near-triangulated and $J$ is a chordless cycle, the graph
$H=G\setminus O(G)$ is connected and near-triangulated, or is empty. 
If $H$ is empty then $G=J$ is a cycle, and $G$ can trivially be $2$-colored 
in such a way that monochromatic components have size at most $2$. 
We may thus suppose that $H$ is not empty. 
Observe that $O(H)=O_2(G)$. By induction, $H$ has a 3-coloring such that

\begin{enumerate}[(i')]
\item no vertex of $O(H)$ is colored with color 1; 
\item no vertex of $O_2(H)$ is colored with color 2;
\item every 2-component intersecting $O(H)$ has size at most
  $f_1(\Delta)$;
\item every 3-component intersecting $O(H)\cup O_2(H)$ has size at
  most $f_2(\Delta)$, and
\item every monochromatic component has size at most $6\Delta^2
  f_2(\Delta)$.
\end{enumerate}

Our aim now is to extend this coloring of $H$ to one of $G$ by
coloring the vertices of $O(G)$ using colors 1 and 2. Let $G'$ be the
graph obtained from $G$ by removing all vertices of $H$ colored 
with color 3 and 
all monochromatic components of $H$ that are disjoint from $O(H)$, and
contracting each 2-component of $H$ intersecting $O(H)$ into a single
vertex. Note that $G'$ is a plane graph as in
Lemma~\ref{lem:blackbox}, with $S$ the set of contracted 2-components.

Observe that vertices of $G'$ in $O(G')=O(G)$ still have degree at
most $\Delta$, and that vertices in $S$ have degree at most $\Delta
\cdot f_1(\Delta)$ by property (iii') of $H$.  We color $G'$ using
Lemma~\ref{lem:blackbox}. In this coloring, 1-components of $G'$ have
size at most $g_1(\Delta)$, while 2-components of $G'$ have size at
most $g_2(\Delta,\Delta \,f_1(\Delta))$.
 
The coloring of $G'$ induces a coloring of the vertices of $O(G)$
that extends the coloring of $H$ we previously obtained to the graph
$G$. In this coloring of $G$, since no vertex of $O(H)$ is colored
with color 1 by property (i') of $H$, each 1-component intersecting
$O(G)$ has size at most $g_1(\Delta)$ by the previous paragraph, which 
proves (a). Also,
each 2-component of $G$ intersecting $O(G)\cup O_2(G)$ corresponds to
a 2-component of $G'$ of size at most $g_2(\Delta,\Delta
\,f_1(\Delta))$.  Hence, each such 2-component of $G$ has size at most
$f_1(\Delta) \,g_2(\Delta,\Delta \,f_1(\Delta))$, showing (b). Moreover,
3-components intersecting $O_2(G)=O(H)$ have size at most $f_2(\Delta)$ by
(iv'), which proves (c).  Finally, every monochromatic component of
$G$ not considered above has size at most $6\Delta^2 f_2(\Delta)$ by
(v'), which concludes this case.

\medskip

\paragraph{\bf Case 3: All bounded faces of $J$ are triangular}
Let $uv$ be an arbitrarily chosen 
edge of $J$ lying on the boundary of its outerface.  
Let $\phi(u)$ and $\phi(v)$ be colors for $u$ and $v$, respectively,  
arbitrarily chosen among 1 and 2. We show that 
$G$ has a 3-coloring satisfying (i)--(v) and the following three 
extra properties:  
\begin{enumerate}
\item each monochromatic component of $G$ 
intersecting $O(G)$ is contained in $O(G)$ and has size at most $2\Delta$; 
\item all vertices in $O_{2}(G)$ have color 3, and 
\item  $u$ and $v$ are colored with colors $\phi(u)$ and $\phi(v)$, 
respectively, and moreover no neighbor of $u$ in $V(G) \setminus \{v\}$ 
is colored with color $\phi(u)$. 
\end{enumerate}

 For each bounded face $f$ of $J$, let $H_{f}$ be the subgraph of $G$
 induced by the vertices lying in the 
 proper interior of $f$. As in {\bf Case 2}, these
 graphs $H_{f}$ are either connected and near-triangulated, or are
 empty. Using induction, for each bounded face $f$ of $J$ such that 
 $H_{f}$ is not empty, we color
 $H_{f}$ with colors 1, 2, 3 in such a way that

\begin{enumerate}[(i')]
\item no vertex of $O(H_{f})$ is colored with color 1; 
\item no vertex of $O_2(H_{f})$ is colored with color 2;
\item every 2-component intersecting $O(H_{f})$ has
size at most $f_1(\Delta)$; 
\item every 3-component intersecting $O(H_{f})\cup O_2(H_{f})$ has size at
  most $f_2(\Delta)$, and
\item every monochromatic component has
size at most $6\Delta^2 f_2(\Delta)$, 
\end{enumerate}
and we recolor with color 3 the at most $3\Delta$ vertices of $O(H_f)$
(that is, the vertices of $H_{f}$ that are adjacent to some vertex in the
boundary of $f$).  Next, we color $u$ and $v$ with colors $\phi(u)$
and $\phi(v)$, respectively, and color the remaining vertices of $J$
according to the parity of their distances to $\{u,v\}$ in $J$: we use
color $\phi(u)$ if the distance is even, and the color opposite to
$\phi(u)$ if it is odd. (As before, the color opposite to 1 is
2, and vice versa.)

Clearly, the resulting coloring of $G$ satisfies (2) and (3). Also,
each monochromatic component $K$ of $G$ that includes a vertex of some
graph $H_{f}$ is contained in $H_{f}$.  The bounds on the size of $K$ are
then guaranteed by (i')--(v'), except possibly in the case where $K$
is a 3-component intersecting $O(H_{f})$. In that case, since we
recolored with color 3 at most $3\Delta$ vertices of $H_{f}$, using
(iv') and Observation~\ref{obs:recolor} (with $\ell=3\Delta$) we
obtain that $K$ has at most $6\Delta^2 f_2(\Delta)$ vertices, as
desired.  Hence, properties (ii)--(v) are satisfied for monochromatic
components of $G$ avoiding $O(G)$. Since the remaining monochromatic
components are contained in $J$, and since we only used colors 1 and 2
when coloring that graph, it only remains to establish property (1).

Consider thus a monochromatic component $K$ of $J$.  First suppose
that $K$ contains $v$. Here there are two possibilities: either
$\phi(u) = \phi(v)$, in which case $V(K) = \{u,v\}$, or $\phi(u) \neq
\phi(v)$, in which case all vertices in $V(K) \setminus \{v\}$ are
neighbors of $u$ or $v$. Note that $|V(K) \setminus \{v\}| \leq 2\Delta -2$ 
in the latter case since $uv\in E(G)$. Hence
$|K| \leq 2\Delta$ holds in both cases. 

Now assume that $K$ avoids the vertex $v$. 
Then by the definition of our coloring, all vertices in $K$ 
are at the same distance $i$ from $\{u,v\}$ in $J$. 
If $i=0$ then $V(K) = \{u\}$, and (1) trivially holds, so assume $i>0$. 
Let $X$ be the set of vertices of $J$ 
at distance $i-1$ from $\{u,v\}$ and having a 
neighbor in $K$. If $|X| \geq 3$, then considering 
the union of three shortest paths from $u$ to three distinct vertices in $X$ 
together with the connected subgraph $K$, we deduce that $J$ contains 
$K_{2,3}$ as a minor. However, this contradicts the fact that $J$ 
is outerplanar. Hence, we must have $|X| \leq 2$, and therefore 
$K$ has at most $|X|\cdot \Delta \leq 2\Delta$ vertices, 
showing (1). This concludes {\bf Case 3}.  

\medskip

Before proceeding with the final case, we need to
introduce some terminology.  First, note that each bounded face of $J$
is bounded by a cycle of $J$ (since $J$ is outerplanar), and that each
vertex of $J$ is in the boundary of at least one bounded face of $J$
(since $G$ has minimum degree at least $2$). In particular, every such
vertex is contained in a cycle of $J$. These basic properties will be used implicitly in what follows.

Since neither {\bf Case 2} nor {\bf Case 3} applies, $J$ has at least
two bounded faces, and at least one of them is not triangular.  For a
bounded face $f$ of $J$, let $G_{f}$ denote the subgraph of $G$
induced by the union of the vertices in the boundary of $f$ and the
vertices lying in the proper interior of $f$.  

We define a rooted tree $\calT$ whose vertices are the bounded faces
of $J$: First, choose arbitrarily a bounded face of $J$ and make it
the root of $\calT$. The tree $\calT$ is then defined inductively as
follows.  If $f$ is a vertex of $\calT$ then its children in $\calT$
are the bounded faces $f'$ of $J$ that are distinct from the parent of
$f$ in $\calT$ (if $f$ is not the root), and such that the boundaries
of $f$ and $f'$ intersect in a non-empty set $X_{f'}$ of vertices
which separates $G_{f} \setminus X_{f'}$ from $G_{f'} \setminus
X_{f'}$ in $G$. The set $X_{f'}$ is then said to be the
\emph{attachment} of the face $f'$. Observe that, since $J$ is
outerplanar, $X_{f'}$ consists either of a single vertex or of two
adjacent vertices. For definiteness, we let the attachment of the root
of $\calT$ be the empty set.

A bounded face $f$ of $J$ is \emph{bad} if
$f$ is triangular and $|X_{f}|=2$, otherwise $f$ is \emph{good}. 
Observe that, in particular, the root of $\calT$ is good.

\medskip

\paragraph{\bf Case 4: None of the previous cases applies}
Let $f$ be a good face maximizing its depth in $\calT$.  Thus all
strict descendants of $f$ in $\calT$ are bad (if any). Let
$\calT_0,\calT_1,\ldots,\calT_k$ be the trees resulting from the
removal of $f$ in $\calT$, where $\calT_0$ contains the parent of $f$
if $f$ is not the root and is otherwise empty, 
and each $\calT_i$ ($i\in\{1,\dots,k\}$) 
contains a different child of $f$ in $\calT$.  (Note that possibly
$k=0$ if $f$ is not the root.)  
Let $X_0$ denote the attachment of $f$, and for each
$i\in\{1,\dots,k\}$ let $X_i$ denote the attachment of the unique
child of $f$ contained in $\calT_i$. By the choice of $f$, each $X_i$
($i\in\{1,\dots,k\}$) consists of two adjacent vertices $u_i,v_i$ of
the boundary of $f$, and at least one of them, say $v_{i}$, is not in
$X_{0}$.

For each $i\in\{0,\dots,k\}$, let 
$$G_i:=G[\cup_{f'\in V(\calT_i)} V(G_{f'})].$$ Notice that, for each
$i\in\{1,\dots,k\}$, all bounded faces of $G_i$ are triangular.  Let
us also recall once again 
that $G_{0}$ is the empty graph in case $f$ is the root
of $\calT$ (in which case $X_{0}$ is empty as well).

We proceed in three steps. 

\smallskip

\noindent \emph{Step 1.} We start by coloring $G_0$ using
the induction (if $G_{0}$ is not empty), 
and $G_f$ using {\bf Case 2} of the induction,
so the resulting coloring of $G_f$ satisfies also (a)--(c).

\smallskip

\noindent \emph{Step 2.} We recolor three sets of vertices of $G_f$:
First, recolor in $G_f$ the vertices of $X_0$ to match their color in
$G_0$.  Next, recolor the at most two vertices in $O(G_f) \setminus
X_0$ having a neighbor in $X_0$ with a color (1 or 2) distinct from
the color of their unique neighbor in $X_0$. Note that the latter can
be done precisely because $f$ is good. (Indeed, either $|X_0|\le 1$,
or $|X_0|=2$ and the cycle bounding the outerface of $G_f$ has length
at least $4$.)  Finally, recolor with color 3 all vertices in $G_f
\setminus O(G_f)$ having a neighbor in $X_0$ (note that there are at
most $2\Delta-4$ such vertices).

\smallskip

\noindent \emph{Step 3.} For each $i\in\{1,\dots,k\}$, color $G_{i}$ using
          {\bf Case 3} of the induction, choosing respectively $u_i$
          and $v_i$ as $u$ and $v$ in (3), and $\phi(u_i)$ and
          $\phi(v_i)$ as the colors of $u_i$ and $v_i$ after Step~2
          above. Recall that $v_i \notin X_0$. 

\smallskip

We claim that the coloring of $G$ obtained by taking the union of the
colorings of $G_f, G_0,\ldots,G_k$ satisfies
(i)--(v). First we remark that, because of the recoloring of $X_0$ at
Step 2 and the use of property (3) in Step 3, the colorings of $G_f,
G_0,\ldots,G_k$ coincide on the pairwise intersection of the vertex
sets of these graphs, so the union of these colorings is well defined.

After Steps 2 and 3, no vertex in $X_0$ is colored with a color that
is used for some of its neighbors in $G_f \setminus X_0$ (after Step
3, this follows from properties (2) and (3)). This implies that every
monochromatic component of $G$ intersecting $V(G_0)$ is contained in
$V(G_0)$, and therefore satisfies (i)--(v) by induction. Similarly,
monochromatic components intersecting $V(G_i)$ but avoiding $X_i$ for
some $i \in \{1,\dots, k\}$ satisfy (i)--(v) by the induction. Hence,
it only remains to consider monochromatic components of $G$ avoiding
$G_0$ (and in particular, avoiding $X_0$) and intersecting $V(G_f)$.

Since we recolored at most two vertices of $O(G_f) \setminus X_0$, it
follows from Observation~\ref{obs:recolor} (with $\ell=2$) that the
size of 1-components after Step~2 of the graph $G_{f}$ intersecting $O(G_f)
\setminus X_0$ is at most $4\Delta$ times the maximum
size of 1-components of $G_f$ intersecting $O(G_f)$ before we
recolored vertices of $G_f$, which was at most $g_1(\Delta)$ by
property (a) from {\bf Case~2}. Now, a 1-component $K$ of the graph
$G$ which intersects $O(G_{f})\setminus X_{0}$ is the union of a {\em
  single} 1-component $K'$ of $G_f$ after Step~2  
intersecting $O(G_f)\setminus
X_{0}$ with at most $2|K'|$ 1-components from the graphs
$G_{1}, \dots, G_{k}$ (since every vertex of $V(K')\cap O(G_f)$ lies
in at most two such graphs).  It follows from (1) in {\bf Case 3} and
the observation above that $|K| \leq 2 |K'| \cdot 2\Delta \le 16
\Delta^{2} g_1(\Delta)$. This proves (iii).

Similarly, using property (b) from {\bf Case 2} we deduce that after
Step~2, 2-components of $G_{f}$ intersecting $(O(G_f) \setminus
X_{0})\cup O_2(G_f)$ have size at most $4\Delta \cdot f_1(\Delta)
g_2(\Delta,\Delta \,f_1(\Delta))$. Applying the same reasoning as for
1-components of $G$ intersecting $O(G_{f})\setminus X_{0}$ above, we
deduce that 2-components of $G$ intersecting $(O(G_f) \setminus
X_{0})\cup O_2(G_f)$ have size at most $16\Delta^2 \cdot f_1(\Delta)
g_2(\Delta,\Delta \,f_1(\Delta))$, which proves (iv).

Finally, using property (c) from {\bf Case 2} and the fact that at
most $2\Delta-4\le 2\Delta$ vertices of $O_{2}(G_f)$ have been
recolored with color 3 in Step~2, we have that 3-components of $G$
intersecting $O_2(G_f)$ have size at most $4\Delta^{2} f_2(\Delta)
\leq 6\Delta^{2}f_2(\Delta)$ by Observation~\ref{obs:recolor} (with
$\ell=2\Delta$).  Therefore, (v) also holds.
\end{proof}

From Theorem~\ref{thm:main} we easily derive our main theorem, 
Theorem~\ref{thm:main_intro}, with an explicit bound. 

\begin{cor}
\label{cor:main}
Every plane graph $G$ with maximum degree $\Delta \geq 1$ can be
$3$-colored in such a way that
\begin{enumerate}[(i)]
\item each monochromatic component has size at most
  $(15\Delta)^{32\Delta+8}$;
\item only colors $1$ and $2$ are used for vertices on the outerface; 
\item each $1$-component intersecting $O(G)$ is included in $O(G)$ and
  has size at most $6^4\,\Delta^3$.
\end{enumerate} 
\end{cor}

\begin{proof}
If $G$ is not connected we can color each component of $G$ separately,
so we may suppose that $G$ is connected.  We may further assume that
$\Delta\ge 3$ since otherwise $G$ is properly 3-colorable. If $G$ is
not near-triangulated, we do the following for every bounded face $f$
of $G$. Let $x_1,x_2,\ldots,x_k$ be a boundary walk of $f$ (note that
a vertex appears at least twice in the walk if and only if it is a
cut-vertex of $G$). We add a cycle $u_1,u_2,\ldots,u_k$ of length $k$
inside $f$ and link each vertex $u_i$ to $x_i$ and $x_{i-1}$ (indices
are taken modulo $k$). Next, for each $i \in \{1, \dots, \lceil k/2
\rceil-1\}$ we add the edges $u_iu_{k-i}$ and $u_{i}u_{k-i+1}$ (if
they are not already present). The graph obtained is near-triangulated
and every new vertex has degree at most 6.  For every original vertex
$v$ of $G$, we added at most two edges incident to $v$ in between 
every two consecutive original edges in the cyclic 
ordering of the edges around $v$. 
Hence the maximum degree of the new graph is at most
$\max(6,3\Delta)\le 3\Delta$ and the result follows from
Theorem~\ref{thm:main} (with $\Delta$ replaced by $3\Delta$).
\end{proof}

\section{Extension to surfaces of higher genus}\label{sec:surfaces}

In this section we extend our main result to graphs embeddable in a
fixed surface.  In this paper, a {\DEF surface} is a non-null compact
connected 2-manifold without boundary. Recall that the {\DEF Euler
  genus} of a surface $\Sigma$ is $2 - \chi(\Sigma)$, where
$\chi(\Sigma)$ denotes the Euler characteristic of $\Sigma$.  We refer
the reader to the monograph by Mohar and Thomassen~\cite{MoharThom} for
basic terminology and results about graphs embedded in surfaces.

Let $f(\Delta)=(15\Delta)^{32\Delta+8}$ be the bound on the size of monochromatic
components in Corollary~\ref{cor:main}.

\begin{thm}
\label{thm:genus}
Every graph $G$ with maximum degree $\Delta \geq 1$ 
embedded in a surface $\Sigma$ of Euler genus $g$  
can be $3$-colored in such a way
that each monochromatic component has size at most
$(5\Delta)^{{2^{g}}-1} f(\Delta)^{2^{g}}$.
\end{thm}
\begin{proof}
The proof proceeds by induction on $g$. If $g=0$, then $G$ is planar
and the result follows from Corollary~\ref{cor:main}. Assume now that
$g>0$. 

We may suppose that some cycle of $G$ is not contractible (as a closed 
curve on the surface), since otherwise 
$G$ can be embedded in the plane. Let $C$ be a shortest non-contractible 
cycle of $G$. If $C$ has a chord $e$, then at least one of the two cycles 
obtained from $C$ using the edge $e$ is not contractible, as follows from the 
so-called 3-Path Property (see~\cite[p.\ 110]{MoharThom}). However, this 
contradicts the minimality of $C$.  Thus the cycle $C$ is induced. 

Each connected component of $G':=G\setminus V(C)$ can be embedded in 
a surface of Euler genus strictly less 
than $g$ (see~\cite[Chapter 4.2]{MoharThom}). Thus, applying induction
on each connected component of $G'$, we deduce that $G'$ can be
$3$-colored in such a way that each monochromatic
component has size at most $s= (5\Delta)^{{2^{g-1}}-1}
f(\Delta)^{2^{g-1}}$.

Let $t:=|C|$. We extend the coloring of $G'$ obtained above to a
coloring of $G$ by coloring the vertices of $C$ as follows.  We divide
them into $k$ circular intervals $I_1,\ldots,I_k$ (where the circular
ordering is of course given by $C$), each of length $s+1$, except
$I_1$ whose length is $t$ if $t \leq s$, and $s+1+(t \pmod {s+1})\le
2s+1$ if $t > s$.  We color all vertices in $I_1$ with color 1, and
for each $i \in \{2, \dots, k\}$, we color vertices in $I_i$ with
color 2 if $i$ is even, and color 3 if $i$ is odd.

If some monochromatic component $K$ of $G'$ has a neighbor $u$ in some
interval $I_i$ and another neighbor $v$ in an interval $I_j$ with
$i\neq j$ that are colored the same as $K$, then one can find a path
$P$ from $u$ to $v$ having all its internal vertices in $K$, and thus
being internally disjoint from $C$. Recall that $|K| \leq s$, and that
by our coloring of the intervals, there are at least $s+1$ vertices
between $u$ and $v$ on both sections of the cycle $C$. Hence, the two
cycles obtained by shortcutting $C$ using the path $P$ are shorter
than $C$.  However, by the 3-Path Property, at least one of them is
not contractible, contradicting our choice of $C$.

It follows that each monochromatic component of $G'$ has neighbors in
at most one interval $I_i$ in the graph $G$.  Using
Observation~\ref{obs:recolor} (with $\ell = 2s+1$), we deduce that
monochromatic components of $G$ have size at most $$2 (2s+1)\Delta
\cdot s\le 5s^{2}\Delta \le 5\Delta \cdot (5\Delta)^{{2^{g}}-2} \cdot
f(\Delta)^{2^{g}} = (5\Delta)^{{2^{g}}-1} f(\Delta)^{2^{g}},$$ as
desired.
\end{proof}

We remark that using the cutting technique introduced recently by
Kawarabayashi and Thomassen~\cite{KT12} together with the
stronger property from Corollary~\ref{cor:main} that one color can be
omitted on the outerface, it is possible to obtain a bound that is
linear in the genus (instead of doubly exponential). We only sketch
the proof in the remainder of this section (we preferred to present
the full details of the simple and self-contained proof of
Theorem~\ref{thm:genus}, at the expense of a worst bound).

Kawarabayashi and Thomassen~\cite[Theorem~1]{KT12} proved that any
graph $G$ embedded on some surface of Euler genus $g$ with
sufficiently large facewidth (say more than $10t$, for some constant
$t$) has a partition of its vertex set in three parts $H,A,B$, such
that $A$ has size at most $10tg$, $B$ consists of the disjoint union
of paths that are local geodesics\footnote{In the sense that each
  subpath with at most $t$ vertices is a shortest path in $G$.} and are
pairwise at distance at least $t$ in $G$, and $H$ induces a planar
graph having a plane embedding such that the only vertices of $H$
having a neighbor in $A\cup B$ lie on the outerface of $H$. 

Recall that by Corollary~\ref{cor:main} every plane graph of maximum
degree $\Delta$ can be colored with colors $1,2,3$ so that no vertex
of the outerface is colored $3$ and each monochromatic component has
size at most $f(\Delta)=(15\Delta)^{32\Delta+8}$. We now prove by
induction on $g$ that for every graph $G$ of Euler genus $g$ and maximum
degree $\Delta$ there is a set of at most $10(f(\Delta)+2)\, g$
vertices in $G$ whose removal yields a graph that has a 3-coloring
where each monochromatic component has size at most $\Delta
f(\Delta)+1$. Using Observation~\ref{obs:recolor}, this will directly imply
that $G$ has a 3-coloring in which every monochromatic component has
size at most $f(\Delta)+20\Delta (f(\Delta)+2)(\Delta f(\Delta)+1)\,
g$, a bound that is linear in $g$.

If $g=0$ the graph is planar and we can apply
Corollary~\ref{cor:main}. If the facewidth is at most
$10(f(\Delta)+2)\, g$, we remove the vertices intersecting a noose of
length at most $10 (f(\Delta)+2)$, and apply induction on the
resulting graph (since each of its components can be embedded 
in a surface of Euler genus at most $g-1$). 
If the facewidth is more than $10(f(\Delta)+2)\, g$ we
apply the result of Kawarabayashi and Thomassen. Let $H,A,B$ be the
corresponding partition of $G$ (with $H$ having its specific plane
embedding). The set $A$ is the set of vertices we remove from $G$. We
now color $H$ using Corollary~\ref{cor:main}, avoiding color 3 on its
outerface.  Recall that each component of $B$ is a path.  For each
such path $P$, choose arbitrarily an endpoint $v$ of $P$ and color all
the vertices of $P$ with color 3, except $v$ and the vertices whose
distance to $v$ in $P$ is a multiple of $f(\Delta)+2$. The latter
vertices are colored with color 2. It can easily be checked that
monochromatic components of color 1 have size at most $f(\Delta)$ and
monochromatic components of color 3 have size at most
$f(\Delta)+1$. Note that every monochromatic component of color 2 in $H$
has at most one neighbor colored 2 in $B$, since otherwise two paths
of $B$, or two vertices that are at distance $f(\Delta)+2$ on
some path of $B$, would be at distance at most $f(\Delta)+1$ in
$G$. Hence every monochromatic component of color 2 has size at most
$\Delta f(\Delta)+1$ in $G$, as desired.

\section{Conclusion}
\label{sec:Conclusion}

We proved that planar graphs with maximum degree $\Delta$ can be
$3$-colored in such a way that each monochromatic
component has size at most $f(\Delta)=(15\Delta)^{32\Delta+8}$.  It is 
thus natural to look for lower bounds on the best possible value for
$f(\Delta)$. The examples constructed in~\cite{KMRV97}
and~\cite{ADOV03} give a lower bound of $\Omega(\Delta^{1/3})$ 
(see also a related construction in~\cite{LMST08}).  We remark that 
this bound can be slightly improved as follows.  
Let $k\geq 3$ and let $G_k$ be the
graph obtained from a path $P$ on $k$ vertices $v_1,\ldots,v_k$ by
adding, for each $i\in \{2,\dots, k\}$, a path $P_i$ on $k(2k-3)$ new
vertices, and making all of them adjacent to $v_{i-1}$ and $v_i$. Note
that this graph is planar and has maximum degree
$\Delta=2k(2k-3)+2$. Consider any 3-coloring of $G_k$. We now prove that
there is a monochromatic component of size at least 
$k=\Omega(\sqrt{\Delta})$. If the path
$P$ itself is not monochromatic, then there exists $j\in \{1, \dots,
k-1\}$ such that $v_j$ and $v_{j+1}$ have distinct colors, say 1 and
2.  If color 1 or color 2 appears $k-1$ times in $P_j$ then we have a
monochromatic star on $k$ vertices. Otherwise there is a subpath of
$P_j$ with $k$ vertices, all of which are colored with color 3.

\medskip

As mentioned in the introduction, Alon, Ding, Oporowski, and
Vertigan~\cite{ADOV03} proved that for every proper minor-closed class
of graphs $\mathcal G$ there is a function $f_{\mathcal G}$ such that
every graph in $\mathcal G$ with maximum degree $\Delta$ can be
4-colored in such way that every monochromatic component has size at
most $f_{\mathcal G}(\Delta)$. On the other hand, for every $t$, there
are graphs with no $K_{t}$-minors that cannot be colored with $t-2$
colors such that all monochromatic components have bounded size. So in
this case again, the assumption that the size depends on $\Delta$
cannot be dropped. We ask whether Theorem~\ref{thm:main_intro} holds
not only for graphs of bounded genus, but more generally for all
proper minor-closed classes of graphs.

\begin{quest}
Is it true that for each proper minor-closed class of graphs $\mathcal
G$ there is a function $f_{\mathcal G}:\N \to \N$ such that every
graph in $\mathcal G$ with maximum degree $\Delta$ can be 3-colored in
such way that each monochromatic component has size at most
$f_{\mathcal G}(\Delta)$?
\end{quest}

Note that the example of graphs with no $K_{t}$-minors that cannot
be colored with $t-2$ colors in such a way that all monochromatic
components have bounded size shows that the famous Hadwiger
Conjecture, stating that graphs with no $K_t$-minor have a proper
coloring with $t-1$ colors, is best possible even if we only ask the
sizes of monochromatic components to be bounded by a function of $t$
(instead of being of size $1$).  On the other hand, Kawarabayashi and
Mohar~\cite{KM07} proved the existence of a function $f$ such that
every $K_t$-minor-free graph has a coloring with $\lceil\tfrac{31}2
t\rceil$ colors in which each monochromatic component has size at most
$f(t)$. This bound was recently reduced to
$\lceil\tfrac72\,t-\tfrac32\rceil$ colors by Wood~\cite{Woo10}. This
is in contrast with the best known bound of $O(t \sqrt{ \log t})$
colors for the Hadwiger Conjecture (see~\cite{Kos84,Tho84}).

\medskip

A well-known result of Gr\"otzsch~\cite{G59} asserts that
triangle-free planar graphs are 3-colorable. A natural question is
whether there exists a constant $c$ such that every triangle-free
planar graph can be 2-colored such that every monochromatic component
has size at most $c$. The following construction shows that the answer
is negative. Fix an integer $k\geq2$ and consider a path $x_1,\ldots,
x_{k}$. For each $i \in \{1, \dots, k\}$, add a set $S_i$ 
of $2k-3$ vertices which are adjacent to $x_i$ only, and finally add
a vertex $u$ adjacent to all vertices in $\bigcup_{1\le i\le k}
S_i$. This graph $G_k$ is planar and triangle-free. Take a 2-coloring
of $G_k$ and assume that the path $x_1,\ldots,x_k$ is not
monochromatic. Then some vertex $x_i$ has a color distinct from that
of $u$. Since $u$ and $x_i$ have $2k-3$ common neighbors, one of $u$
and $x_i$ has $k-1$ neighbors of its colors, and then lies in a
monochromatic component of size $k$. It follows that in every 2-coloring
of $G_k$ there is a monochromatic component of size at least $k$. Note
that this construction has unbounded maximum degree. Hence, the
following natural question remains open.

\begin{quest}
Is there a function $f:\N \to \N$ such that every triangle-free planar graph
with maximum degree $\Delta$ can be 2-colored in such a way that each 
monochromatic component has size at most $f(\Delta)$?
\end{quest}


\begin{thebibliography}{99}

\bibitem{ADOV03} N. Alon, G. Ding, B. Oporowski, and D. Vertigan,
  \emph{Partitioning into graphs with only small components},
  J. Combin. Theory Ser. B {\bf 87} (2003), 231--243.

\bibitem{B08} R. Berke, \emph{Coloring and Transversals of Graphs},
  PhD thesis, ETH Zurich, 2008. Diss. ETH No. 17797.

\bibitem{G59} H.~Gr\"{o}tzsch, {\em Ein Dreifarbensatz f\"ur dreikreisfreie
  Netze auf der Kugel}, Wiss. Z. Martin-Luther-Univ. Halle-Wittenberg
  Math.-Natur. Reihe {\bf 8} (1959) 109--120.


\bibitem{HST03} P. Haxell, T. Szab\'o, G. Tardos, \emph{Bounded size
  components--partitions and transversals}, J. Combin. Theory Ser. B
  {\bf 88} (2003), 281-- 297.


\bibitem{KM07} K. Kawarabayashi and B. Mohar, \emph{A relaxed
  Hadwiger'€™s conjecture for list colorings}, J. Combin. Theory Ser. B
  {\bf 97(4)} (2007) 647--651.

\bibitem{KT12} K.~Kawarabayashi and C.~Thomassen, \emph{From the plane
  to higher surfaces}, J. Combin. Theory Ser. B {\bf 102} (2012),
  852--868.

\bibitem{Kos84} A.~Kostochka, \emph{Lower bound of the Hadwiger number
    of graphs by their average degree}, Combinatorica {\bf 4} (1984),
  307--316.

\bibitem{KMRV97} J. Kleinberg, R. Motwani, P. Raghavan, and
  S. Venkatasubramanian, \emph{Storage management for evolving
    databases}, Proceedings of the 38th Annual IEEE Symposium on
  Foundations of Computer Science (FOCS 1997), 353--362.

\bibitem{LMST08} N. Linial, J. Matou\v sek, O. Sheffet, and G. Tardos,
  \emph{Graph coloring with no large monochromatic components},
  Combin. Prob. Comput. {\bf 17(4)} (2008), 577--589.

\bibitem{LS93} N. Linial and M. Saks, \emph{Low diameter graph
  decompositions}, Combinatorica {\bf 13} (1993), 441--454.

\bibitem{MoharThom}
B. Mohar and C. Thomassen.
\newblock {\em Graphs on surfaces}.
\newblock Johns Hopkins University Press, Baltimore, U.S.A., 2001.

\bibitem{Tho84} A.~Thomason, \emph{An extremal function for
    contractions of graphs}, Math. Proc. Cambridge Philos. Soc. {\bf
    95} (1984), 261--265.

\bibitem{Woo10} D.R. Wood, \emph{Contractibility and the Hadwiger
  Conjecture}, European J. Combin. {\bf 31(8)} (2010), 2102--2109.


\end{thebibliography}
\end{document}